\documentclass{amsart}
\usepackage{amsmath,amsfonts,amssymb,amsthm, esint}
\usepackage[english]{babel}
\usepackage{enumerate}
\usepackage{dsfont}
\usepackage{comment}

\usepackage[colorlinks,citecolor=blue]{hyperref}
\usepackage{epsfig,graphicx} 

\newtheorem{theorem}{Theorem}[section]
\newtheorem{definition}[theorem]{Definition}

\newtheorem{lemma}[theorem]{Lemma}

\newtheorem{corollary}[theorem]{Corollary}

\theoremstyle{remark}
\newtheorem{remark}[theorem]{Remark}
\newtheorem{example}[theorem]{Example}

\numberwithin{equation}{section}
\numberwithin{figure}{section}

\newcommand{\R}{\mathbb{R}}
\newcommand{\N}{\mathbb{N}}

\def\supp{\mathrm{supp}\,}

\begin{document}
\title[Chow and Rashevskii meet Sobolev]{Chow and Rashevskii meet Sobolev}
\author{Sergey Kryzhevich}
\address[Sergey Kryzhevich]{
	Institute of Applied Mathematics, Faculty of Applied Physics and Mathematics, Gdańsk University of Technology, 80-233 Gdańsk, Poland
	\and
	BioTechMed Center, Gdańsk University of Technology 
}
\email[Sergey Kryzhevich]{kryzhevich@gmail.com}
\author{Eugene Stepanov}
	\address{
		Dipartimento di Matematica, Universit\`a di Pisa,
		Largo Bruno Pontecorvo 5, 56127 Pisa, Italy
		\and 
		St.Petersburg Branch of the Steklov Mathematical Institute of the Russian Academy of Sciences,
	St.Petersburg, 
	Russian Federation
	\and
	HSE University, Moscow, Russian Federation		
}
\email[Eugene Stepanov]{stepanov.eugene@gmail.com}

	\author{Dario Trevisan}
\address{Dario Trevisan, Dipartimento di Matematica, Universit\`a di Pisa \\
	Largo Bruno Pontecorvo 5 \\ I-56127, Pisa}
\email{dario.trevisan@unipi.it}

\dedicatory{The paper is dedicated to Giovanni Alberti in the occasion of his 60th birthday}

\thanks{	
	E.S.\ and D.T.\ acknowledge the MIUR Excellence Department Project awarded to the Department of Mathematics, University of Pisa, CUP I57G22000700001.	
	The work of E.S.\ is partially within the framework of HSE University Basic Research Program.
	D.T.\ acknowledges the project  G24-202 ``Variational methods for geometric and optimal matching problems'' funded by Università Italo Francese,  the HPC Italian National Centre for HPC, Big Data and Quantum Computing --  CUP I53C22000690001, the PRIN 2022 Italian grant 2022WHZ5XH - ``understanding the LEarning process of QUantum Neural networks (LeQun)'', CUP J53D23003890006, the INdAM-GNAMPA project 2025 ``Analisi spettrale, armonica e stocastica in presenza di potenziali magnetici''.  Research also partly funded by PNRR - M4C2 - Investimento 1.3, Partenariato Esteso PE00000013 - "FAIR - Future Artificial Intelligence Research" - Spoke 1 "Human-centered AI", funded by the European Commission under the NextGeneration EU programme. S.K. acknowleges Peking University and Sichuan University for the invitation and financial support. 
}

\date{\today}

\begin{abstract}
We prove a weak version of the Chow-Rashevskii theorem for vector fields having only Sobolev regularity and generating suitable flows
as selections of solutions to the respective ODEs, for a.e.\ initial datum.
\end{abstract}

\keywords{Chow-Rashevskii theorem, H\"{o}rmander condition, geometric control, global controllability, Sobolev spaces, Regular Lagrangian Flows}

\maketitle

\section{Introduction}

The celebrated Chow-Rashevskii theorem~\cite{BoscainSigalotti19-control}  (also known as Chow-Rachevsky theorem) plays an important role in both differential geometry and in control theory. It says, roughly speaking,  that if a family of smooth vector fields over some smooth manifold 
satisfies the so-called H\"{o}rmander or  Lie bracket generating condition,
then each point of the manifold can be reached from any other point in finite time by only following the flows of vector fields of the family. By smoothness one means usually $C^\infty$ regularity (although of course just finite times continuous differentiability suffices). It is therefore quite natural to ask whether a similar result holds for much less regular vector fields which may have even no classical derivatives.

The above question is quite nontrivial even if we weaken the regularity requirement only slightly. Consider for instance the simplest case of two vector fields $X_1$ and $X_2$ in $\R^3$. The H\"{o}rmander condition in this case with both $X_1$ and $X_2$ smooth, is guaranteed to be fulfilled, if, for instance,  the vectors
$X_1(x)$, $X_2(x)$, $[X_1, X_1](x)$, form a basis of $\R^3$ for every $x\in \R^3$, where $[X_1, X_1]$ stands for the Lie bracket of $X_1$ and $X_2$. If we try to weaken just slightly the smoothness condition, assuming $X_1$ and $X_2$ to be only Lipschitz (of course
in this case they still define uniquely the respective flows for every initial datum, so the thesis of the Chow-Rashevskii theorem would be still meaningful), there is a problem in extending the H\"{o}rmander condition. In fact, for smooth $X_1$ and $X_2$ their
Lie bracket  $[X_1, X_2]$ is defined by the formula
\begin{equation}\label{eq_Lie1}
	[X_1, X_2](x) := DX_2(x) X_1(x)- DX_1(x) X_2(x),
\end{equation}
where $DX_i$ are Jacobi matrices of $X_i$, $i=1,2$, and for $X_1$ and $X_2$ only Lipschitz  this formula makes sense 
only for almost every $x\in \R^3$ with respect to the Lebesgue measure, since the derivatives of Lipschitz functions
may only be guaranteed to exist almost everywhere.  On the other hand the ``almost everywhere'' version of the  H\"{o}rmander condition
(i.e.\ just requiring that $X_1(x)$, $X_2(x)$, $[X_1, X_1](x)$, form a basis of $\R^3$ for only a.e.\ $x\in \R^3$) is clearly insufficient for the thesis of the Chow-Rashevskii theorem to hold even for smooth vector fields. For instance if $X_1$ and $X_2$ are smooth and $X_1$, $X_2$, $[X_1, X_1]$ are all tangent to some sphere, then the flows of $X_1$ and $X_2$  will never exit this sphere, hence points from the ball bounded by this sphere can never reach the outside of this ball following the flows of  $X_1$ and $X_2$. A possible remedy would be introducing some
sort of quantitative version of the H\"{o}rmander condition quantifying the nondegeneracy of the triple of vectors $X_1(x)$, $X_2(x)$, $[X_1, X_1](x)$.

Besides smoothness or Lipschitz continuity, there are many other regularity properties of the vector field $V\colon \R^d\to \R^d$
guaranteeing the existence and uniqueness of solutions to the Cauchy problem for the ODE
\begin{equation}\label{eq_chowFrobCauchy1}
	\dot{x}=V(x), \quad x(0)=y
\end{equation} 
for every or just for a.e.\ initial datum $y\in \R^d$, and hence defining a flow $\varphi_V$ by the formula  $\varphi_V(t,y):=x(t)$, where
$x(\cdot)$ is a solution to~\eqref{eq_chowFrobCauchy1}.
For instance, as shown in~\cite[corollary 5.2]{caravenna2018directional}, this is the case of Sobolev vector fields  $V\in W^{1,p}(\R^d;\R^d)$ with $p>d$ and bounded divergence  $\mathrm{div}\, V$, with the flow defined uniquely for a.e.\ initial datum.
Moreover, even if $p\leq d$, then a solution to~\eqref{eq_chowFrobCauchy1} may be not unique for a.e.\ initial datum,
but in this case there is a natural selection of such solutions called {\em regular Lagrangian flow}~\cite{DiPernaLions89, Ambrosio04_transpBV}. 
It is also well-known that there are many other cases besides Sobolev regularity and bounded divergence when the vector field $V$ admits a regular Lagrangian flow \cite{crippa2008estimates}.

In this paper we prove a very weak version of the Chow-Rashevskii theorem for vector fields having only Sobolev regularity and locally bounded divergence, once they generate  reasonable flows
as selections of solutions to the respective ODEs for a.e.\ initial datum. This includes (though is not limited to)
the above mentioned cases
of Sobolev vector fields generating regular Lagrangian flows. We show that whenever the vector fields satisfy the natural quantitative version of the  H\"{o}rmander condition (which quantifies how much nondegenerate should be the basis of the space made by the original vector fields and their consecutive Lie brackets of any order at a.e.\ point of the space) plus a mild growth condition providing existence of the flows globally in time, then one can arrive not from any point to any point  (like in the classical smooth case), but rather from an arbitrarily small  neighborhood of any point to a.e.\ point in the space.

To  prove the announced result one has to overcome two principal difficulties. The first is that the flows of Sobolev vector fields are known to be wildly discontinuous, in particular, they are just summable but not even Sobolev~\cite{Jabin-nonreg16}, so that their weak derivatives do not have any pointwise (even a.e.) meaning. This corresponds well to the recently established loss of regularity results for solutions of continuity equations with Sobolev velocity fields of bounded divergence or even divergence-free~\cite{ACM19,BQ19}.  Even when the vector fields are Lipschitz, although  it is a textbook result that their respective flows are Lipschitz too, but no more than that. This makes impossible any direct extension of the classical proof of the Chow-Rashevskii theorem based on pointwise Taylor expansions of the flows. The second, though a bit hidden, difficulty is that the Lie bracket of Sobolev vector fields does not have an easy geometric meaning of an infinitesimal commutator of the flows. For instance, vanishing of the Lie bracket of any couple of smooth vector fields implies the commutativity of the flows of the latter, and this is not the case if the vector fields are less regular, see  \cite[example~3.1]{RigSteTrev22-commutSobolev}. However, it has been shown both in~\cite{RigSteTrev22-commutSobolev} and in~\cite{ColombTione22-commutSobolev} that this is still the case when just one of the vector fields is Lipschitz. That is why the technique we develop here is essentially different from the classical one  and is based  upon a PDE/measure theory-style approach, which may be interesting in itself.

\section{Notation and preliminaries}

\subsection{Basic notation}
For two real numbers $a$ and $b$ we denote $a\wedge b$ the minimum of them. 

\subsection{Spaces} 
The $d$-dimensional Euclidean space $\R^d$ is always assumed to be endowed with the
Euclidean norm  $|\cdot|$ and by the Lebesgue measure $\mathcal{L}^d$.
We denote by $L^p(\Omega;\R^m)$ (resp.\ $L^p_{loc}(\Omega;\R^m)$) 
the usual Lebesgue space of 
integrable (resp.\ locally integrable) with exponent $p\geq 1$
maps $f\colon \Omega \subset \R^d\to \R^m$ (essentially bounded when $p=+\infty$), where $\Omega$ is a Lebesgue measurable subset of $\R^d$.
The canonical norm in $L^p(\Omega;\R^m)$ is denoted by $\|\cdot \|_{p, \Omega}$, the reference to $\Omega$ may be omitted when there is no possibility of confusion. 
Analogously, the notation $W^{k,p}(\R^d;\R^m)$ (resp.\ $W^{k,p}_{loc}(\R^d;\R^m)$) will stand for the usual Sobolev  (resp.\ locally Sobolev) class of maps over $\R^d$ with values in $\R^m$.
The notation $C^\infty_0(\R^d;\R^m)$  stands for the class of infinitely differentiable functions with compact support in $\R^d$ (usually called test functions) with values in $\R^m$ and the action of a (vector valued) distribution $u$ on a test function $\varphi\in C^\infty_0(\R^d;\R^m)$ is denoted by $\langle \varphi, u\rangle$.
In all the cases the reference to $\R^m$ will be omitted when $m=1$, i.e. for real valued functions. 

For $p\in [1,\infty]$ we denote by $p'$ the usual conjugate exponent $1/p'+1/p=1$, and by $p^*$ the critical Sobolev embedding exponent in $\R^d$ defined as
\[
p^*:=\left\{
\begin{array}{rl}
	\frac{pd}{d-p},& 1\leq p<d,\\
	+\infty, & p>d.
\end{array}
\right.
\]
(for $p:=d$ we assume $p^*$ to be an arbitrary number in $[1,+\infty)$).

\subsection{Sets and functions} For a set $E\subset \R^d$ we let $\mathrm{diam}\, E$ stand for the diameter of $E$, $\mathbf{1}_E$ stand for the characteristic function of $E$, $(E)_\varepsilon$ stand for its $\varepsilon$-neighborhood, i.e.\ the set of points with distance from $E$ strictly less than $\varepsilon>0$, by $\bar E$ the closure of $E$. The notation $B_R(0)$ stands for the open ball in $\R^d$ of radius $R$ centered at the origin.

When dealing with Borel functions $f$ and $g$ defined on a metric space $E$ with Borel measure $\mu$ we write
$f\leq g$ (resp.\ $f= g$), if $f(x)\leq g(x)$ (resp.\ $f(x)= g(x)$, for $\mu$-a.e.\ $x\in E$. If  $E_1$ and $E_2$ are metric spaces,
$T\colon E_1\to E_2$ a Borel function, $\mu$ a Borel measure on $E_1$, we define $T_{\#}\mu$
the pushforward measure of $\mu$ on $E_2$  by
$(T_{\#}\mu)(B):=\mu(T^{-1}(B))$, where $B$ is a Borel set,  and for a Borel function $f\colon E_2\to \R$ we denote by $T_*f$ its pullback to $E_1$ defined by $(T_*f)(x):= f(T(x))$ for every $x\in E_1$.

\subsection{Vector fields and their Lie brackets}\label{sec_Liebrack1}
For a Sobolev vector field $V\in W^{1,p}_{loc}(\R^d;\R^d)$ we denote by $DV$ its weak (distributional) Jacobi matrix (i.e.\ the matrix of its distributional derivatives) and by $\mathrm{div} V$ its distributional divergence (which is in fact the trace of the Jacobi matrix). 
	
Let $X_j\colon \R^d\to \R^d$, $j=1,\ldots, k$ be vector fields which will be further assumed to be Sobolev, i.e.
$X_j\in W^{1,p}_{loc}(\R^d;\R^d)$ for some $p\geq 1$.
Denote for brevity 
\[
\mathcal{V}:=\{X_1,\ldots, X_k\}.
\]
We define then inductively the sets of vector fields $\mathcal{V}_m$, $m\in \N$ 
as follows. 

\begin{definition}\label{def_bracksets1}
Let 
\begin{itemize}
	\item[(i)] $\mathcal{V}_0:=\mathcal{V}$. In this case, for an $X\in  \mathcal{V}_0$ we set $\mathcal{P}(X):=\{X\}$,
	\item[(ii)] $\mathcal{V}_m$, for $m>0$, be the set of vector fields
	$Y\colon \R^d\to \R^d$ such that 
	\[Y=[U,V] :=DV\cdot U -DU\cdot V \] 
	for some Sobolev vector fields $U$ and $V$ with $U\in W^{1,q}_{loc}(\R^d;\R^d)$, $\mathrm{div} U\in W^{1,q}_{loc}(\R^d)$,  
	and $V\in W^{1,q'}_{loc}(\R^d;\R^d)$, $\mathrm{div} V\in W^{1,q'}_{loc}(\R^d)$ (the exponent $q\in (1,+\infty)$ possibly depending on the pair $(U,V)$)
	with either $U\in \mathcal{V}_{m-1}$, $V\in \mathcal{V}_j$, or vice versa $U\in \mathcal{V}_j$, $V\in \mathcal{V}_{m-1}$, $j=1,\ldots, m-1$. In this case  $\mathcal{P}(Y):= \mathcal{P}(U)\cup \mathcal{P}(V)$.
\end{itemize}
In this way,  the set $\mathcal{V}_\infty:=\cup_k \mathcal{V}_k$ is the set of all $X_j$ and all their Lie brackets of all orders
(if the latter may be defined in the sense of the above~(ii)).
For an $Y\in \mathcal{V}_\infty$ we refer to $\mathcal{P}(Y)\subset \{X_1,\ldots, X_k\}$ as the set of all vector fields from
$\{X_1,\ldots, X_k\}$ involved in the construction of $Y$.
\end{definition}

\begin{remark}
	Note that in the above Definition~\ref{def_bracksets1} the sets 
	$\mathcal{P}(Y)$ are not uniquely defined. For instance, if $\mathcal{V}_0= \{X_1, X_2, X_3\}$ with  $X_3= [X_1,X_2]$,
	then both $X_3\in \mathcal{V}_0$, in which case $\mathcal{P}(X_3)= \{X_3\}$, and $X_3\in \mathcal{V}_1$, in which case $\mathcal{P}(X_3)= \{X_1, X_2\}$. In other words, $\mathcal{P}(Y)$ 
	may be different for different constructions of $Y\in \mathcal{V}_\infty$ according to Definition~\ref{def_bracksets1}. However in 
	what follows we  need to know $\mathcal{P}(Y)$ only for some chosen construction of the latter.
\end{remark}

\begin{example}\label{ex_invE41}
	If $\mathcal{V}_0:=\{X_1,X_2\}$, then $\mathcal{V}_1:=\{[X_1,X_2]\}$ once $X_1\in W^{1,q}_{loc}(\R^d;\R^d)$, $\mathrm{div} X_1\in W^{1,q}_{loc}(\R^d)$,  
	and $X_2\in W^{1,q'}_{loc}(\R^d;\R^d)$, $\mathrm{div} X_2\in W^{1,q'}_{loc}(\R^d)$. The latter, as in Remark~\ref{rm_invE41},
	holds true when, for instance $\{X_1,X_2\}\subset W^{1,q}_{loc}(\R^d;\R^d)$ and $\{\mathrm{div} X_1, \mathrm{div} X_2\} \subset W^{1,q}_{loc}(\R^d)$ for some $q\ge 2$, which is, in its turn, guaranteed when both $X_1$ and $X_2$ belong to $W^{2,q}_{loc}(\R^d;\R^d)$, $q\geq 2$.
\end{example}

\subsection{Flows}
Any smooth vector field $V\colon \R^d\to\R^d$, for which the solutions to the Cauchy problem for the ODE
\begin{equation}\label{eq_flow1}
	\dot x = V(x),\quad x(0)=y
\end{equation}
are defined uniquely and globally in time, 
generates the flow defined by the map
$\varphi_V(t,y):= x(t)$ where $x(\cdot)$ is the (unique) solution to~\eqref{eq_flow1}. We define now a weaker notion of the flow adapted for 
possibly discontinuous vector fields for which one can indicate a reasonable selection of solutions to~\eqref{eq_flow1} at least for almost every (in the sense of the Lebesgue measure $\mathcal{L}^d$) initial datum $y\in\R^d$. 
Namely, we will give the following definition.

\begin{definition}\label{def_flowgen1}
We  say that a locally integrable vector field $V\colon \R^d\to\R^d$ is (weak) flow generating, if there exists
a map $\varphi_V\colon \R\times \R^d\to \R^d$ such that
\begin{itemize}
	\item[(i)] $\varphi_V$ is a Carath\'{e}odory function, in the sense that $\varphi_V(\cdot,y)$ is continuous for a.e.\
	$y\in \R^d$, and $\varphi_V(t, \cdot)$ is Borel for every $t\in \R$,
	\item[(ii)] for a.e.\ 	$y\in \R^d$, the curve $x(\cdot):=\varphi_V(\cdot,y)$ is absolutely continuous and the ODE $\dot{x}(t)=V(x(t))$ is
	satisfied for a.e.\ $t\in \R$, with initial datum $x(0)=y$,
	\item[(iii)] one has $\varphi_V(t,\varphi_V(s,y))=\varphi_V(t+s, y)$ for a.e.\
	$y\in \R^d$ and every $t,s\in \R$, with the Lebesgue nullset where the identity fails  possibly depending on $t$ and $s$,
	\item[(iv)] the measures $\mu_t:= \varphi_V(t,\cdot)_{\#}\mathcal{L}^d$ are absolutely continuous with respect to the Lebesgue measure
	$\mathcal{L}^d$, i.e.\ $\mu_t\ll \mathcal{L}^d$ for all $t\in  \R$,
	\item[(v)]   there exist smooth vector fields $V_k\colon \R^d\to\R^d$  approximating $V$ by convolution with a compactly supported approximate identity, so that $V_k\to V$ in $L^\mathbf{1}_{loc}(\R^d;\R^d)$, and generating the classical flows $\varphi_{V_k}$, such that
	for  the measures $\mu_t^k:= \varphi_{V^k}(t,\cdot)_{\#}\mathcal{L}^d$ one has
	$\mu_t^k\rightharpoonup \mu_t$ in the weak sense of measures as $k\to\infty$ for a.e.\ $t\in \R$,
	 \item[(vi)] if $W\colon \R^d\to\R^d$ satisfy $V(x)=W(x)$ for a.e.\ $x\in \R^d$, then $W$ us also flow generating with
	 $\varphi_V(t,y))=\varphi_W(t,y)$ for a.e.\ $y\in \R^d$ and all $t\in \R$.   
\end{itemize}
\end{definition}

\begin{remark}\label{rm_RLF1}
The above definition is satisfied in particular by a Sobolev vector field $V\in W^{1,1}_{loc}(\R^d;\R^d)$ with bounded divergence
$\mathrm{div}\, V\in L^\infty(\R^d)$. In this case, when $V$ satisfies some rather mild growth condition, 
then it generates a flow in the sense of the above definition, and for such a flow one may take the so-called ~\textit{regular Lagrangian flow}
of $V$.
In fact, the existence of a regular Lagrangian flow for such a vector field 
is given by~\cite[Theorem 4.8]{Ambrosio_Crippa_2014}. The properties~(i),~(ii) of the above Definition~\ref{def_flowgen1} are just in the definition of a regular Lagrangian flow, the property~(iv) is even weaker than the one necessarily satisfied by regular Lagrangian flows,
the group property~(iii) is~\cite[Remark 4.9c)]{Ambrosio_Crippa_2014} and the stability property~(v) for regular Lagrangian flows follows from the general stability result~\cite[Theorem 3.13]{Ambrosio_Crippa_2014}, or, alternatively, from the more precise
result for Sobolev fields \cite[Theorem 4.10]{Ambrosio_Crippa_2014}
(which, however, require a uniform bound on the vector field, but it is far from being
necessary). Finally, the property~(vi) for regular Lagrangian flows, somehow folkloric, can be found, e.g. in section 9.5 of~\cite{Ambrosio_Crippa_2014}. 
\end{remark}

We will further omit the word ``weak'' and call the respective vector fields just flow generating, similarly to smooth vector fields.
Instead of $\varphi_V(t,y)$ we will always write $e^{tV}(y)$. 

\section{Setting and main result}
Let $B\subset \R^d$ be an arbitrary Borel set.
Consider the family of Borel sets
\[
\mathcal{F}  := \left\{
e^{t_{j_1} X_{j_1}}\circ\ldots\circ  e^{t_{j_m} X_{j_m}} B \colon  
j_i\in \{1,\ldots, k\}, (t_{j_1}, \ldots, t_{j_m})\in \R^m , m\in \N
\right\},
\]
and define $f_B\colon \R^d\to \R$ to be the measure-theoretic supremum 
\[
f_B:= \bigvee_{D\in \mathcal{F}} \mathbf{1}_D
\]
of the characteristic functions
$\mathbf{1}_D$ among all $D\in \mathcal{F}$, that is the lowest (with respect to the a.e.\ pointwise inequality) Lebesgue measurable function $f$ which is greater or equal than all $\mathbf{1}_D$ for all $D\in \mathbb{E}$. In yet another words $f_B(x)\geq \mathbf{1}_D(x)$
for all $D\in \mathcal{F}$, and for any
function $g$ with the same property one has $f_B(x)\leq  g(x)$, all these inequalities holding for a.e.\
$x\in \R^d$ (the sets of points where these inequalities do not hold, though having all zero Lebesgue measure, may be different for different inequalities). The following obvious lemma says that $f_B$ is in fact a characteristic function.

\begin{lemma}\label{lm_reacable_set1}
	One has $f_B(x)\in \{0,1\}$ for a.e.\
	$x\in \R^d$.
\end{lemma}

\begin{proof}
	Clearly $f_B(x)\geq  0$ for a.e. 
	$x\in \R^d$ being greater than all the characteristic functions. One also has $f_B(x)\leq  1$ for a..e. 
	$x\in \R^d$ since otherwise $f_B\wedge 1$ would give a lower than $f_B$ function majorizing all $\mathbf{1}_D$ for $D\in \mathcal{F}$. It remains to observe that the set $T:=\{x\in \R^d\colon 0<f_B(x)<1)\}$ cannot have positive Lebesgue measure, since otherwise this  would mean
	$\mathbf{1}_D(x)=0$ for a.e.\ $x\in E$ and all $D\in\mathcal{F}$, and therefore $f_B\wedge \mathbf{1}_{E^c}$ would still majorize all $\mathbf{1}_D$ for $D\in \mathcal{F}$ but is lower than $f_B$.
\end{proof}

In view of Lemma~\ref{lm_reacable_set1} one may write $f_B= \mathbf{1}_E$ for some Borel $E=E(B)\subset \R^d$.
The goal of the present paper is proving an extension of the Chow-Rashevskii theorem assuming the following weak version of the 
classical H\"{o}rmander condition adapted for Sobolev vector fields, which we call \textit{Sobolev-H\"{o}rmander condition}.

\begin{definition}\label{def-SobolHorm1}
	We say that the Sobolev-H\"{o}rmander condition holds, if
	there exist vector fields $Y_j\in \mathcal{V}_m$ for some $m\geq 0$ depending on $Y_j$, with 
	$Y_j\in W^{1,q}_{loc} (\R^d;\R^d)$, $\mathrm{div} Y_j \in L^r_{loc} (\R^d)$, $r>1$, $j=1,\ldots, d$, 
	such that 
	\begin{itemize}
		\item[(i)] 	the $d\times d$ matrix  $Y(x)$, the rows of which are vectors $Y_j(x)$, is invertible
		for a.e.\ $x\in \R^d$ with the matrix function
		$Y^{-1}\colon \R^d\to \R^{d\times d}$ satisfying $Y^{-1}\in  W^{1,s}_{loc} (\R^d;\R^{d\times d})$ for some $s>(q^*\wedge r)'$, while 
		\item[(ii)] all
		the vector fields $X_k\colon \R^d\to \R^d$ involved in the construction of  $Y_j$ (i.e.\ $X_k\in \mathcal{P}(Y_j)$ in the sense of Definition~\ref{def_bracksets1}) are flow generating,  
		and 
		$\mathrm{div} X_k\in W^{1,1}_{loc}(\R^d)\cap L^\infty_{loc} (\R^d)$. 
	\end{itemize}
\end{definition}

\begin{example}\label{ex_invE42}
	If $d=3$, suppose that the vector fields $X_1$ and $X_2$ satisfy  $\{X_1,X_2\}\subset W^{1,2}_{loc}(\R^3;\R^3)$ and $\{\mathrm{div} X_1, \mathrm{div} X_2\} \subset W^{1,2}_{loc}(\R^3)\cap  L^\infty_{loc} (\R^3)$, 
    $[X_1,X_2]\in  W^{1,p}_{loc}(\R^3;\R^3)$ with $ \mathrm{div} [X_1,X_2]\in  L^r_{loc} (\R^3)$.
    Set $q:= p\wedge 2$. Then for the Sobolev-H\"{o}rmander condition to hold it is sufficient that 
	the $3\times 3$ matrix  $Y(x)$, the rows of which are vectors $X_1(x)$, $X_2(x)$, $[X_1, X_2](x)$, be invertible
	for a.e.\ $x\in \R^3$ with the matrix function
	$Y^{-1}\colon \R^3\to \R^{3\times 3}$ satisfying $Y^{-1}\in  W^{1,s}_{loc} (\R^3;\R^{3\times 3})$ for some $s>(q^*\wedge r)'$.
	
	In particular, if $\{X_1,X_2\}\subset W^{2,2}_{loc}(\R^3;\R^3)$, then automatically  
	\[
	\{X_1,X_2\}\subset L^\infty_{loc}(\R^3;\R^3) \quad\mbox{and}\quad 
	\{D X_1,DX_2\}\subset L^{6}_{loc}(\R^3;\R^{3\times 3}) 
	\]
	 by Sobolev embedding theorem, which implies
	\[[X_1,X_2]\in  W^{1,3/2}_{loc}(\R^3;\R^3).\]
	Taking then $q:=3/2$ we get $q^*=3$.
	Further, 
	\[
	\mathrm{div}[X_1,X_2] = X_1\cdot \nabla \mathrm{div}X_2- X_2 \cdot \nabla \mathrm{div}X_1 \in L^2_{loc}(\R^3),
	\]
	and hence if $\{\mathrm{div} X_1, \mathrm{div} X_2\} \subset L^\infty_{loc} (\R^3)$, 
	then one can take in the above $s>2$ unless for some reason $\mathrm{div}[X_1,X_2]\in L^r_{loc}(\R^3)$ for some $r\geq 3$, in which case $s> 3'=3/2$ or for some $r\in (2,3)$, in which case $s> r'$.
\end{example}

We are now at a point to announce the principal result of this paper. 

\begin{theorem}\label{th_chow1weak}
	If all $X_k$ are sublinear in the sense
	\[
	|X_k(x)|\leq \alpha_k |x|+\beta_k \quad\mbox{for a.e.\ $x\in \R^d$}
	\]
	for some $\alpha_k>0$, $\beta_k>0$,
	and the Sobolev-H\"{o}rmander condition holds, then 
	\begin{itemize}
		\item[(i)] 	For every Borel $B\subset\R^d$ of positive Lebesgue measure one has 
		\[\mathcal{L}^d(\R^d\setminus E(B))=0.\]
	\end{itemize}
	In particular,  
	\begin{itemize}
		\item[(ii)] for every couple of points $\{x,y\}\in \R^d$ and every $\varepsilon >0$ there is an $x_\varepsilon\in B_\varepsilon(0)\subset \R^d$ an $m\in \N$ and an $m$-tuple $(t_1,\ldots, t_m)\in \R^m$ of instances of time 
		and $m$ indices 
		\[
		j_i\in \{1,\ldots, k\}, \qquad i=1,\ldots, k,
		\]
		such that 
		\[
e^{t_{j_1} X_{j_1}}\circ  e^{t_{j_m} X_{j_m}} x_\varepsilon \in B_\varepsilon(y). 
		\]
		In simple words, this says that for every couple of points one may start arbitrarily close to the first one and
		arrive arbitrarily close to the second one following only the flows of the vector fields $X_k$,
	\end{itemize}	
\end{theorem}

\begin{remark}\label{rm_locHormander1}
In the above Theorem~\ref{th_chow1weak} one might weaken the requirement on the validity of the Sobolev-H\"{o}rmander condition in the sense of Definition~\ref{def-SobolHorm1}. Namely, the thesis of this theorem remains valid if one requests instead of the Sobolev-H\"{o}rmander condition only its localized version, namely, that 
there exist
 a family $\mathcal{A}$ of open connected sets covering all $\R^d$ 
	such that for every $\Omega\in \mathcal{A}$ there exist
	vector fields $Y_j\in L^1_{loc} (\R^d;\R^d) \cap W^{1,q}_{loc} (\Omega;\R^d)$, $Y_j\in \mathcal{V}_\infty$, $j=1,\ldots, d$, 
such that the $d\times d$ matrix  $Y(x)$, the rows of which are vectors $Y_j(x)$, is invertible
for a.e.\ $x\in \Omega$ with the matrix function
$Y^{-1}\colon \Omega\to \R^{d\times d}$ satisfying $Y^{-1}\in  W^{1,s}_{loc} (\Omega;\R^{d\times d})$ for some $s>q'$
(both $q$ and $s$ may depend on $\Omega$), while all
the vector fields $X_k\colon \R^d\to \R^d$ involved in $Y_j$ are flow generating,  
and 
$\mathrm{div} X_k\in W^{1,1}_{loc}(\R^d)\cap L^\infty_{loc} (\R^d)$. 
\end{remark}

\begin{remark}\label{rm_Horlassic1}
	The classical H\"{o}rmander condition on the set of \textit{smooth} vector fields $\{X_1,\ldots, X_k\}$ is that of existence 
	for every $x\in \R^d$ 
	of
	$Y_j\in \mathcal{V}_\infty$, $j=1,\ldots, d$, such that the vectors $Y_1(x),\ldots, Y_d(x)$ form a basis of $\R^d$.
	It is worth noting that under this condition the localized Sobolev-H\"{o}rmander condition from Remark~\ref{rm_locHormander1} necessarily holds, since the $d\times d$ matrix  $Y(x)$, the rows of which are vectors $Y_j(x)$, is invertible, and the set $\Omega\subset \R^n$ of $x\in \R^d$ where it is invertible is open, while
the matrix function
	$Y^{-1}$ is smooth (so in particular, $Y^{-1}\in  W^{1,\infty}_{loc} (\Omega;\R^{d\times d})$). 
\end{remark}

\begin{remark}\label{rm_locHormander2}
	It is worth emphasizing that both the Sobolev-H\"{o}rmander condition as formulated in Definition~\ref{def-SobolHorm1}, and its slightly more general localized version of Remark~\ref{rm_locHormander1} implicitly request higher Sobolev regularity of the vector fields once their Lie brackets are involved (see section~\ref{sec_Liebrack1}).
\end{remark}


\subsection{An ``almost classical'' toy application}

To formulate a toy application, for a Borel set $B\subset \R^d$ we denote 
\[\mathcal{R}(B):=\bigcup_{D\in\mathcal{F}}D.\] 
In control theory the latter  is usually called the ``attainable set'' from $B$, i.e.\ the set of points where one could arrive using all the possible flows along vector fields from $\mathcal{V}$ starting at a point in $B$. In particular, $\mathcal{R}(\{x\})$ is the set attainable from the given point $x$. For an $X\in \mathcal{V}$ and an $x\in \R^d$ such that the flow $e^{tX}(x)$ is undefined we set without loss of generality $e^{tX}x:=x$ for all $t\in \R$, so that the flows become formally defined everywhere. 
The following lemma is immediate.

\begin{lemma}\label{lm_Rmeas1}
If $\mathcal{R}(B)$ is Lebesgue measurable, then $\mathbf{1}_{\mathcal{R}(B)}=\mathbf{1}_{E(B)}$ a.e.
(or, in other words, $\mathcal{R}(B)=E(B)$ up to a Lebesgue negligible set).
\end{lemma}

\begin{proof}
	One has $\mathbf{1}_D\leq \mathbf{1}_{\mathcal{R}(B)}$ for all $D\in \mathcal{F}$, hence
	$\mathbf{1}_{\mathcal{R}(B)}\geq \mathbf{1}_{E(B)}$ a.e.\ On the other hand, since $\mathbf{1}_D\leq \mathbf{1}_{E(B)}$ for all $D\in \mathcal{F}$ and $\sup_{D\in \mathcal{F}}\mathbf{1}_D= \mathbf{1}_{\mathcal{R}(B)}$ is Lebesgue measurable , then
	\[
	\mathbf{1}_{\mathcal{R}(B)}= \sup_{D\in \mathcal{F}}\mathbf{1}_D\leq \mathbf{1}_{E(B)},
	\]
	which concludes the proof.
\end{proof}

We now are able to provide the following corollary of our main Theorem~\ref{th_chow1weak}.

\begin{corollary}\label{co_contr1}
	Suppose that for every compact set $K\subset \R^d$ the set $\mathcal{R}(K)$ is $\sigma$-compact. 
 Then under the conditions of Theorem~\ref{th_chow1weak} (or of the Remark~\ref{rm_locHormander1}) one has the following
  alternative:
\begin{itemize}
	\item[(i)] either there is a set $D\subset\R^d$ of full Lebesgue measure in $\R^d$ such that for every $x\in D$ one has $\mathcal{R}(\{x\})=D$,
	\item[(ii)] or $\mathcal{R}(\{x\})$ has zero Lebesgue measure for all $ x\in \R^d$.
\end{itemize}
%
\end{corollary}

\begin{proof}
Since $\mathcal{R}(K)$ is $\sigma$-compact for every compact $K\subset \R^d$, then
it is so also for every $\sigma$-compact $K$.
In particular $\mathcal{R}(\{x\})$ is $\sigma$-compact for every $x\in \R^d$. 
If there is a point $x\in \R^d$ such that $\mathcal{R}(\{x\})$ has nonzero Lebesgue measure, then 
\begin{align*}
	\mathcal{R}(\{x\}) & = \mathcal{R} (\mathcal{R}(\{x\}))\\
 & =E (\mathcal{R}(\{x\}) \quad\mbox{by Lemma~\ref{lm_Rmeas1}}\\
&  =\R^d \quad\mbox{by Theorem~\ref{th_chow1weak}}, 
\end{align*}
the last two inequalities being intended up to a Lebesgue-nullset. 
In other words, the set $D:=\mathcal{R}(\{x\})$ has full Lebesgue measure in $\R^d$.
Now, if  $y\in \R^d$ is any point for which $\mathcal{R}(\{y\})$ has nonzero Lebesgue measure, then the latter set
is also of full Lebesgue measure in $\R^d$, and therefore, $\mathcal{R}(\{x\})\cap \mathcal{R}(\{y\})\neq \emptyset$, and hence,
$\mathcal{R}(\{y\})=\mathcal{R}(\{x\})=D$, concluding the proof.
%
\end{proof}

\begin{remark}\label{rem_classChow1}
	An immediate consequence of Corollary~\ref{co_contr1} is the classical Chow-Rashevskii theorem for smooth vector fields.
In fact, the classical H\"{o}rmander condition for smooth vector fields  implies the validity of the conditions of Remark~\ref{rm_locHormander1}. Moreover,
when all vector fields in $\mathcal{X}$ are smooth, 
then
$\cup_{t\in [-T,T]} e^{tX}(K)$ is compact for every $X\in \mathcal{X}$, $T\in \R$ and hence,  $\mathcal{R}(K)$ is $\sigma$-compact. 
But by Krener's theorem (theorem~2 from~\cite{Krener74}) $\mathcal{R}(\{x\})$ contains an open set for every $x\in \R^d$,
and therefore Corollary~\ref{co_contr1} implies $\mathcal{R}(\{x\})=\R^d$  for all $x\in \R^d$, which is the claim of the Chow-Rashevskii theorem. Note that Krener's theorem used here does not involve anything of the machinery of the classical proof of the
Chow-Rashevskii theorem (namely it does not use Taylor formulae for the flows); it is based only on the immediate observation that
if two smooth vector fields are tangent to some smooth submanifold of $\R^d$ in a neighborhood of some point, then so is also its
Lie bracket (in the same neighborhood). Summing up, we see that  Theorem~\ref{th_chow1weak} and, henceforth, Corollary~\ref{co_contr1} provide an alternative proof of the classical Chow-Rashevskii theorem.
\end{remark}

\begin{remark}
For the Krener's theorem referenced in the above Remark~\ref{rem_classChow1} to be valid, even the classical H\"{o}rmander condition is unnecessarily strong. It is enough in fact to require a weaker condition, namely, the for every point $p\in \R^d$  in the phase space, every open neighborhood $U\subset \R^d$ of this point, and every smooth submanifold $N\subset \R^d$ of strictly positive codimension
containing $p$, there is an $x\in N\cap U$ and an $X_i\in \mathcal{X}$ such that the vector $X_i(x)$ is not tangent to $N$.  
\end{remark}

\section{Proof ot Theorem~\ref{th_chow1weak}}

The proof of the above Theorem~\ref{th_chow1weak} will require several steps. For brevity, we write $E:=E(B)$.

\subsection{Basic idea of the proof}

We explain here the basic idea of the proof in simple terms, which in fact make it more wishful thinking than
a rigorous argument. The latter, however, will be developed in the following subsections in full details following
these basic ideas.

To prove  Theorem~\ref{th_chow1weak} we first show that the set $E=E(B)$ is invariant under the flow of every vector field
$X\in \{X_1,\ldots, X_k\}$. This means that the 
density $\mathbf{1}_E$
satisfies the stationary transport equation
\[
X\cdot \nabla \mathbf{1}_E=0.
\]
We then show that $X\cdot \nabla \mathbf{1}_E=0$ and $Y\cdot \nabla \mathbf{1}_E=0$ imply
\[
[X,Y]\cdot \nabla \mathbf{1}_E=0.
\] 
This is quite intuitive once we think of $X\cdot \nabla \mathbf{1}_E$ in the differential geometry style as the action $X\mathbf{1}_E$ of the vector field $X$ on the function $\mathbf{1}_E$: in fact, once $X\mathbf{1}_E=0$ and $Y\mathbf{1}_E=0$, then
\[
[X,Y]\mathbf{1}_E= X(Y\mathbf{1}_E)-Y(X\mathbf{1}_E)=0.
\] 

Now, as a consequence we get by induction that 
\[
Y_j\cdot \nabla \mathbf{1}_E=0, \qquad j=1,\ldots, d
\]
for all vector fields $Y_j$ in the statement of the Sobolev-H\"{o}rmander condition.
In the matrix form, this is written as
\[
Y\nabla \mathbf{1}_E=0,
\]
 where $Y$ is the $d\times d$ matrix with columns $Y_j$,  $j=1,\ldots, d$. 
 Since the latter is assumed invertible, this means
 \[
\nabla  \mathbf{1}_E=0,
 \]
 that is,  $\mathbf{1}_E=\mathrm{const}$. The latter constant may be clearly either $0$ or $1$, but since $B\subset E$ and $B$ has nonzero measure, then $\mathbf{1}_E=1$, or, in other words, $E=\R^d$ up to a Lebesgue nullset, which is the claim being proven. 

\subsection{Invariance of $\mathbf{1}_E$}

\begin{lemma}\label{lm_invE1}
     For every $X\in \{X_1,\ldots, X_k\}$ 
     such that $e^{-sX}\mathcal{L}^d\ll \mathcal{L}^d$
     for every $s\in \R$, one has
     \[
     e^{tX} _* \mathbf{1}_E= \mathbf{1}_E
     \]
     $\mathcal{L}^d$-a.e.\ on $\R^d$
     for every $t\in \R$.
\end{lemma}

\begin{remark}
	Note that in the above Lemma we do not request any regularity assumption on the vector fields $X_j$ except that
	they be flow generating. Likewise, in all the subsequent lemmata we list all the requested assumptions on the vector fields involved explicitly in the statements, without referring to any common assumption.
\end{remark}

\begin{proof}
For every $D\in\mathcal{F}$ one has $\mathbf{1}_D= e^{tX}_*\mathbf{1}_{D'}$
for $D':= e^{-tX} D$. But $D'\in \mathcal{F}$, hence
\[
\mathbf{1}_{D'}\leq \mathbf{1}_E,
\]
and therefore by Lemma~\ref{lm_auxineq1} one has
\[
\mathbf{1}_D=e^{tX}_* \mathbf{1}_{D'}\leq e^{tX}_* \mathbf{1}_E.
\]
This means $e^{tX}_* \mathbf{1}_E$ majorizes all $\mathbf{1}_D$, $D\in \mathcal{F}$, which implies 
\begin{equation}\label{eq_Emaj1}
	\mathbf{1}_E\leq e^{tX}_*  \mathbf{1}_E.
\end{equation}
From~\eqref{eq_Emaj1} we have
\[
\mathbf{1}_E= (e^{tX}\circ e^{-tX})_*  \mathbf{1}_E = e^{-tX}_*(e^{tX}_*  \mathbf{1}_E) \geq e^{-tX}_* \mathbf{1}_E
\]
again by Lemma~\ref{lm_auxineq1}, 
and since $t\in \R$ is arbitrary, we finally get
\begin{equation}\label{eq_Emaj2}
	\mathbf{1}_E\geq e^{tX}_*  \mathbf{1}_E.
\end{equation}
The estimates~\eqref{eq_Emaj1} and~\eqref{eq_Emaj2} together give the claim.
\end{proof}

As a direct consequence of Lemma~\ref{lm_invE1} we get the following result.

\begin{lemma}\label{lm_invE2}
	Let $X\colon \R^d\to \R^d$ be flow generating, with 
\begin{equation}\label{eq_subl1}
		\begin{aligned}
		\mathrm{div} X &\in W^{1,1}_{loc}(\R^d)\cap L^\infty_{loc} (\R^d),\quad \mbox{and}\\
	|X(x)|&\leq \alpha |x|+\beta
	\end{aligned}
\end{equation}	for a.e.\ $x\in \R^d$ and for some $\alpha>0$, $\beta>0$. 
		Then, the transport equation
	\begin{equation}\label{eq_X1Ea}
			\mathrm{div} (X \mathbf{1}_E)- \mathbf{1}_E\mathrm{div} X =0 
	\end{equation}
		holds in the sense of distributions, i.e.\
		\[
		-\int_E \nabla \varphi\cdot X \, dx-\int_E \varphi \mathrm{div} X \, dx =0
		\]
		for all $\varphi\in C_0^\infty(\R^d)$.
	In particular, if $X$ is divergence free, i.e.\ $\mathrm{div} X=0$, then
	$\mathrm{div} (X \mathbf{1}_E)=0$ in the sense of distributions, i.e.\
	\[
	\int_E \nabla \varphi\cdot X \, dx=0
	\]
	for all $\varphi\in C_0^\infty(\R^d)$.
\end{lemma}

\begin{proof}
	Fix an arbitrary $\varphi\in C_0^\infty(\R^d)$. Then
	\begin{equation}\label{eq_derflow0}
	\int_E\varphi(x)\, dx = 
	\int_E\varphi(e^{-tX} (e^{tX}(x)))\, dx = 	\int_E\varphi(e^{-tX} (y))\, d\mu_t(y),
\end{equation}
	where 
	$\mu_t:= e^{tX}_\# (\mathcal{L}^d)$ (so that $\mu_0=dx$).
	By~Definition~\ref{def_flowgen1}(vi) we may assume without loss of generality that the sublinear estimate in~\eqref{eq_subl1} holds for all $x\in \R^d$. We further set artificially $e^{tX} (x):=x$ for $x\in \R^d$ for which the flow $e^{tX}$ is undefined.
Thus by Lemma~\ref{lm_Kprime_est2} applied with $X$ in place of $V$ 
	one has $e^{tX} (\supp\varphi)\subset K'$ for all $t\in (-T,T)$ and some compact set $K'\subset \R^d$ 
	 depending only on $\supp\varphi$, $\alpha$, $\beta$ and $T>0$. 
	Therefore we may apply Lemma~\ref{lm_LnC2nonsmooth} with $f(t,x):=\varphi(e^{-tX} (x)) \mathbf{1}_E(x)$, $V:=X$
 to obtain the relationship	
	\begin{equation}\label{eq_derflow1}
		\begin{aligned}
		\int_{E} \varphi(e^{-tX} (x)) \, d\mu_t(x) &= \int_{E} \varphi(e^{-tX} (x)) \, dx - t \int_{E} \varphi(e^{-tX} (x)) \mathrm{div} X(x) \, dx + R(t) ,
	\end{aligned}
	\end{equation}	
	valid for all $t\in (-T,T)$, where $|R(t)|\leq C\|\varphi\|_\infty t^2/2$ for some $C>0$ independent of $t$.
	But
\begin{equation}\label{eq_derflow2}
	\begin{aligned}
		\left.\frac{d\,}{dt}\right|_{t=0}	\int_E 	\left(\varphi(e^{-tX} (x)\right)\, dx  
		=\int_E \left.\frac{d\,}{dt}\right|_{t=0}	\left(\varphi(e^{-tX} (x)\right)\, dx  
		= 		-\int_E \nabla \varphi\cdot X \, dx,
	\end{aligned}
\end{equation}	
the first equality (i.e. the passage of the derivative inside the integral) is due to the Lebesgue dominated convergence theorem.
Analogously,
\begin{equation*}\label{eq_derflow2a}
	\begin{aligned}
		\left.\frac{d\,}{dt}\right|_{t=0}	\int_E 	\left(\varphi(e^{-tX} (x)\right)\mathrm{div} X(x)\, dx  
			= 		-\int_E \nabla \varphi\cdot X \mathrm{div} X\, dx,
	\end{aligned}
\end{equation*}	
and thus
\begin{equation}\label{eq_derflow2b}
	\begin{aligned}
		\left.\frac{d\,}{dt}\right|_{t=0}	\left( t \int_E 	\left(\varphi(e^{-tX} (x)\right)\mathrm{div} X(x)\, dx\right)  
		= 		\int_E \varphi(x)\mathrm{div} X(x)\, dx.
	\end{aligned}
\end{equation}	
Combining~\eqref{eq_derflow1}  with~\eqref{eq_derflow2} and~\eqref{eq_derflow2b}, we arrive at
	\begin{equation}\label{eq_derflow3}
	\begin{aligned}
			\left.\frac{d\,}{dt}\right|_{t=0}		\int_{E} \varphi(e^{-tX} (x)) \, d\mu_t(x) &= 
				-\int_E \nabla \varphi\cdot X \, dx -\int_E \varphi\mathrm{div} X\, dx		
	\end{aligned}
\end{equation}	
Hence from~\eqref{eq_derflow0} using~\eqref{eq_derflow3}, we get
\begin{equation*}
		\begin{aligned}
	0 &=\left. \frac{d\,}{dt}\right|_{t=0}	\left(\int_E\varphi(x)\, dx\right)\\
	& = 
		\left.\frac{d\,}{dt}\right|_{t=0}		\int_{E} \varphi(e^{-tX} (x)) \, d\mu_t(x) 
		= 
		-\int_E \nabla \varphi\cdot X \, dx -\int_E \varphi\mathrm{div} X\, dx	
\end{aligned}
\end{equation*}
as claimed.
\end{proof}

\begin{remark}
	Formally writing 
	\[
		\mathrm{div} (X \mathbf{1}_E)= \mathbf{1}_E\mathrm{div} X + X\cdot \nabla \mathbf{1}_E 
	\]
	as if $X$ and $\mathbf{1}_E$ were smooth and $\mathrm{div}$ were the classical (and not just distributional) divergence operator, we 
	see that~\eqref{eq_X1Ea} becomes the transport equation in its usual form
\begin{equation}\label{eq_divX2}
		X\cdot \nabla \mathbf{1}_E =0,
\end{equation}	
	that is, the action of a vector field $X$ on the characteristic function $\mathbf{1}_E$ is zero.
	Although heuristically justifying the name ``transport equation'' attributed to~\eqref{eq_X1Ea}, this is of course, 
	only purely formal. In fact, a priori $E$ is just a Borel set, and hence
	 the gradient $ \nabla \mathbf{1}_E$ can only be understood in the distributional sense and is just a distribution, not even a measure. However a posteriori, when Theorem~\ref{th_chow1weak} is proven, 
	 it claims actually $\mathbf{1}_E=1$, so that~\eqref{eq_divX2} is satisfied. 
\end{remark}

\subsection{Distributional solutions of the transport equation~\eqref{eq_X1Ea}}

We make a couple of observations regarding Sobolev vector fields $X$ satisfying~\eqref{eq_X1Ea}.

\subsubsection{Properties of solutions}

\begin{lemma}\label{lm_invE2a}
	If $X\in W^{1,q}_{loc} (\R^d;\R^d)$ with $\mathrm{div} X\in L^r_{loc}(\R^d)$, $r>1$,
	satisfies~\eqref{eq_X1Ea} in the sense of distributions, then so does
	$\psi X$ for every $\psi\in C_0^\infty(\R^d)$. 
\end{lemma}

\begin{proof}
	Since $r>1$ by assumption and $q^*>1$, then $(q^*\wedge r)' <\infty$.
Let $u_k$ be smooth functions with $u_k\to \mathbf{1}_E$ in $L^p_{loc}(\R^d)$ with $p\in [\hat q', +\infty)$,
where $\hat q:= q^*\wedge r$ 
(e.g.\ $u_k$ can be obtained from $\mathbf{1}_E$ by means of the customary convolution with an approximate identity). Then
\begin{align*}
r_k &:= \mathrm{div} u_k X - u_k\mathrm{div} X \rightharpoonup 0 
\end{align*}
in the sense of distributions as $k\to\infty$ by Lemma~\ref{lm_invE2a_syst1}. But then $r_k= X\cdot \nabla u_k$ and
\begin{align*}
	\mathrm{div} u_k (\psi X) - u_k\mathrm{div} (\psi X)  = \psi X\cdot \nabla u_k=
	\psi r_k \rightharpoonup 0 
\end{align*}
in the sense of distributions and hence passing to the limit as $k\to\infty$ in the above relationship
we get~\eqref{eq_X1Ea} with 
$\psi X$ instead of $X$ as claimed.
\end{proof}

We used here the first of the assertions of the following simple lemma that will be used at full strength also in the sequel.

\begin{lemma}\label{lm_invE2a_syst1}
	If $X\in W^{1,q}_{loc} (\R^d;\R^d)$ with $\mathrm{div} X\in L^r_{loc}(\R^d)$, $r>1$,  
	satisfies~\eqref{eq_X1Ea} in the sense of distributions
	and $u_k$ be smooth functions with $u_k\to \mathbf{1}_E$ in $L^p_{loc}(\R^d)$ with $p\in [ \hat q',+\infty)$. 
	where $\hat q:= q^*\wedge r$. 
	Then letting
	\begin{align*}
		r_k &:= \mathrm{div} u_k X - u_k\mathrm{div} X,
	\end{align*}
	one has $r_k\rightharpoonup 0$ in the sense of distributions as $k\to\infty$.
Moreover,
	\[
	\lim_k \sup\left\{\langle\varphi, r_k\rangle\colon \varphi\in C_0^\infty(\R^d), \supp\varphi\subset K, \|\varphi\|_{W^{1,s}(\R^d)}\leq 1\right\}=0
	\]
	when $1/s\leq 1-1/p- 1/\hat q$, $K\subset \R^d$ compact. In particular, one has for such $s$ that
	\[
	-\int_{\R^d} u_k \nabla \varphi\cdot X \, dx-\int_{\R^d} u_k \varphi \mathrm{div} X \, dx \to 
	0
	\]
	as $k\to\infty$.
\end{lemma}

\begin{proof}
Let $\varphi_k\in C_0^\infty(\R^d), \supp\varphi_k\subset K, \|\varphi_k\|_{W^{1,s}(\R^d)}\leq 1$ be such that
\[
\langle\varphi_k, r_k\rangle	\geq \sup\left\{\langle\varphi, r_k\rangle\colon \varphi\in C_0^\infty(\R^d), \supp\varphi\subset K, \|\varphi\|_{W^{1,s}(\R^d)}\leq 1\right\}-\frac 1 k.
\]
Since $\supp\varphi_k\subset K\Subset \R^d$, then there is a subsequence of $k$ (not relabeled) such that
$\varphi_k\rightharpoonup\varphi$ in the weak sense of $W^{1,s}(\R^d)$, and therefore, since $u_k\to \mathbf{1}_E$ in $L^p_{loc}(\R^d)$, 
$\mathrm{div} X\in L^r_{loc}(\R^d)$ by the assumptions and $X\in L^{q^*}_{loc}(\R^d)$ by Sobolev embedding theorem,
observing that $1/s+1/q^* +1/p\leq 1$ and  $1/s+1/r +1/p\leq 1$,
we get
\begin{align*}
	\langle\varphi_k, r_k\rangle &= 	-\int_{\R^d} u_k \nabla \varphi_k\cdot X \, dx-\int_{\R^d} u_k \varphi_k \mathrm{div} X \, dx \\
	&\to -\int_E \nabla \varphi\cdot X \, dx-\int_E \varphi \mathrm{div} X \, dx =0.
\end{align*}
Thus 
\begin{align*}
0 &\leq \sup\left\{\langle\varphi, r_k\rangle\colon \varphi\in C_0^\infty(\R^d), \supp\varphi\subset K, \|\varphi\|_{W^{1,s}(\R^d)}\leq 1\right\}\\
& \leq \langle\varphi_k, r_k\rangle+\frac 1 k \to 0
\end{align*}
as $k\to\infty$ concluding the proof.
\end{proof}

Yet another  statement will be important.

\begin{lemma}\label{lm_invE3}
If  $X\in W^{1,q}_{loc} (\R^d;\R^d)$ with $\mathrm{div} X\in L^r_{loc}(\R^d)$, $r>1$, satisfies~\eqref{eq_X1Ea}  in the sense of distributions, then  
\begin{equation}\label{eq_divXphi2}
	-\int_E \nabla \varphi\cdot X \, dx-\int_E \varphi \mathrm{div} X \, dx =0
\end{equation}
for all  $\varphi\in W^{1,s}(\R^d)$ with compact support $s\geq \hat q'$, where $\hat q:= q^*\wedge r$.  
In particular, if $X$ is also divergence free, then
\begin{equation}\label{eq_divXphi1}
	\int_E \nabla \varphi\cdot X \, dx=0
\end{equation}
for all $\varphi\in W^{1,s}(\R^d)$ with compact support, $s\geq (q^*)'$.  
\end{lemma}

\begin{proof}
	If  
$\varphi\in W^{1,s}(\R^d)$, $s\geq \hat q'$, and $\supp\varphi$ is compact, consider an open set $\Omega\subset\R^d$ such that
$\supp\varphi \Subset \Omega$. Then $\varphi\in W^{1,\hat q'}_0(\Omega)$. Recalling that the latter space is the closure
in $W^{1,s}(\Omega)$
of  $C_0^\infty(\Omega) $, we
let
$\varphi_k\in C_0^\infty(\R^d) $ be such that $\varphi_k\to \varphi$ as $k\to \infty$ in  $W^{1,\hat q'}(\R^d)$, that is,
$\varphi_k\to \varphi$ in  $L^{s}(\R^d)$ and $\nabla \varphi_k\to \nabla \varphi$ in  $L^{s}(\R^d;\R^d)$ as $k\to \infty$. 
	Since
\begin{equation}\label{eq_divXphik2}
	-\int_E \nabla \varphi_k\cdot X \, dx-\int_E \varphi_k \mathrm{div} X \, dx =0,
\end{equation}
 then keeping in mind that  $\mathbf{1}_E X\in L^{q^*}_{loc}(\R^d;\R^d)$, $\mathbf{1}_E\mathrm{div}X \in L^r_{loc}(\R^d)$
	and passing to the limit as $k\to\infty$ in~\eqref{eq_divXphik2}, we get~\eqref{eq_divXphi2} concluding the proof of the general case.
If $X$ is divergence free, then it suffices to apply the proven claim with $r:=\infty$. 
\end{proof}

\begin{remark}
	If $X\in W^{1,q} (\R^d;\R^d)$, then we may act as in the proof of the above Lemma~\ref{lm_invE3} for an arbitrary
	$\varphi\in W^{1,q'}(\R^d)$, recalling that $W^{1q'}(\R^d)=W_0^{1,q'}(\R^d)$ and letting
	$\varphi_k\in C_0^\infty(\R^d) $ be such that $\varphi_k\to \varphi$ as $k\to \infty$ in  $W^{1,q'}(\R^d)$, 
	we get from~\eqref{eq_divXphik2} the validity of~\eqref{eq_divXphi2} (in particular,~\eqref{eq_divXphi1}
	for divergence free $X$) for all  	$\varphi\in W^{1,q'}(\R^d)$.
\end{remark}

\subsubsection{Lie brackets}

We prove now that if the transport equation~\eqref{eq_X1Ea} holds for two different vector fields, the it holds also for their Lie bracket.
In the formulation of the respective result we find it convenient to use the ``differential geometry style'' notation $Y\varphi:= Y \cdot \nabla\varphi$
for a vector field $Y$ and a smooth function $\varphi$. The purely ``heuristic intuition'' behind this result is then quite clear:
if $X\varphi=0$ and $Y\varphi=0$, then 
\[
[X,Y]\varphi = X(Y\varphi)-Y(X\varphi)=0.
\]

\begin{lemma}\label{lm_invE4}
	If 
	$Y\in W^{1,q'}_{loc}(\R^d;\R^d)$ with $\mathrm{div} Y \in W^{1,q'}_{loc}(\R^d)$, and $X\in W^{1,q}_{loc}(\R^d;\R^d)$
	with $\mathrm{div} X \in W^{1,q}_{loc}(\R^d)$, $q\in (1,+\infty)$,
	both satisfy the transport equation~\eqref{eq_X1Ea} in the sense of distributions, and $\varphi\in C_0^\infty(\R^d)$, then
\begin{equation}\label{eq_divXphiW2a6}
	-\int_E [X,Y]\varphi  \, dx-\int_E \varphi \mathrm{div}[X,Y] \, dx =0, 
\end{equation}
or, in other words, the Lie bracket $[X,Y]$ satisfies~\eqref{eq_X1Ea} in the sense of distributions.
In particular, if both $X$ and $Y$ are also divergence free, then
	\begin{equation}\label{eq_divXphiW2b}
	\int_E [X,Y]\varphi  \, dx=0.
\end{equation}
\end{lemma}

\begin{remark}\label{rm_invE41}
The regularity requirement of the above Lemma~\ref{lm_invE4} holds when, for instance,
$\{X,Y\}\subset W^{1,q}_{loc}(\R^d;\R^d)$ and $\{\mathrm{div} X, \mathrm{div} Y\} \subset W^{1,q}_{loc}(\R^d)$ for some $q\geq 2$, which is, in its turn, guaranteed when both $X$ and $Y$ belong to $W^{2,q}_{loc}(\R^d;\R^d)$, $q\geq 2$ (although this is only a sufficient, but not necessary condition). In fact, if $q\geq 2$, then $q'\leq 2\leq q$, and hence $Y\in W^{1,q'}_{loc}(\R^d;\R^d)$ with $\mathrm{div} Y \in W^{1,q'}_{loc}(\R^d)$. 
\end{remark}

\begin{proof}
	If
	$Y\in W^{1,q'}_{loc}(\R^d;\R^d)$ and $\varphi\in C_0^\infty(\R^d)$, then $Y\varphi:= Y \cdot \nabla \varphi\in W^{1,q'}(\R^d)$, and has compact support. Clearly once $\mathrm{div} X\in L^r_{loc}(\R^d)$ with some $r\geq q$, then $\hat q:= q^*\wedge r\geq q$, and hence
	$\hat q' \leq q'$. 
	Therefore we may apply
	Lemma~\ref{lm_invE3} with $Y\varphi$ in place of $\varphi$ and $s:=q'$, which gives
	\begin{equation}\label{eq_divXphiW2a}
		-\int_E X(Y\varphi ) \, dx-\int_E (Y \cdot \nabla \varphi)\, \mathrm{div} X \, dx =0.
	\end{equation} 
	Writing
	\[
	Y \cdot \nabla \varphi = \mathrm{div} (\varphi Y) - \varphi \mathrm{div} Y,
	\]
	we transform~\eqref{eq_divXphiW2a} into
\begin{equation}\label{eq_divXphiW2a1}
	-\int_E X(Y\varphi ) \, dx-\int_E \mathrm{div} (\varphi Y)\, \mathrm{div} X \, dx + \int_E \varphi 
	\mathrm{div} (Y)\, \mathrm{div} X \, dx =0.
\end{equation}
Interchanging $X$ with $Y$ in~\eqref{eq_divXphiW2a1} we get
\begin{equation}\label{eq_divXphiW2a2}
	-\int_E Y(X\varphi ) \, dx-\int_E \mathrm{div} (\varphi X)\, \mathrm{div} Y \, dx + \int_E \varphi 
	\mathrm{div} (X)\, \mathrm{div} Y \, dx =0,
\end{equation}
and subtracting~\eqref{eq_divXphiW2a2} from~\eqref{eq_divXphiW2a1} we arrive at
\begin{equation}\label{eq_divXphiW2a3}
	-\int_E [X,Y]\varphi  \, dx+\int_E \mathrm{div} (\varphi X)\, \mathrm{div} Y \, dx - \int_E \mathrm{div} (\varphi Y)\, \mathrm{div} X \, dx =0.
\end{equation}
This immediately gives~\eqref{eq_divXphiW2b} when both $X$ and $Y$ are divergence free.

For the generic case, by Lemma~\ref{lm_invE2a} $ \varphi Y$ satisfies~\eqref{eq_X1Ea} in the sense of distributions.
Thus we may apply Lemma~\ref{lm_invE3} with $\varphi Y\in W^{1,q'}(\R^d)$ instead of $X$,  
$q'$ in place of both $q$ and $r$,  $s:=q$, 
and
$\psi \mathrm{div} X\in W^{1,q}(\R^d)$ instead of $\varphi$, where
$\psi\in C_0^\infty(\R^d)$ is such that $\psi=1$ in a neighborhood of $\supp\varphi$,
to get
\begin{equation}\label{eq_divXphiW2a4}
 \int_E \mathrm{div} (\varphi Y)\, \mathrm{div} X \, dx =  -\int_E \varphi Y\cdot \nabla \mathrm{div}X \, dx.
\end{equation}
Interchanging $X$ and $Y$ in~\eqref{eq_divXphiW2a4} yields 
\begin{equation}\label{eq_divXphiW2a5}
	\int_E \mathrm{div} (\varphi X)\, \mathrm{div} Y \, dx =  -\int_E \varphi X\cdot \nabla \mathrm{div}Y \, dx.
\end{equation}
Now, plugging~\eqref{eq_divXphiW2a4} and~\eqref{eq_divXphiW2a5} into~\eqref{eq_divXphiW2a3}, and recalling that
\[
\mathrm{div}[X,Y] = X\cdot \nabla \mathrm{div}Y- Y\cdot \nabla \mathrm{div}X,
\]
by Lemma~\ref{lm_divLie1},
we finally arrive at~\eqref{eq_divXphiW2a6}.
%
%
%
%
\end{proof}

\subsection{A system of transport equations for $\mathbf{1}_E$}

\begin{lemma}\label{lm_systtransp1b}
	Let $\Omega\subset \R^d$ be an open connected set.
	If $Y_j\in W^{1,q}_{loc} (\Omega;\R^d)$, 
	$\mathrm{div} Y_j\in L^r_{loc} (\Omega;\R^d)$, $r>1$,
	$j=1,\ldots, d$,  
	satisfy~\eqref{eq_X1Ea}, i.e.\
	\[
	\mathrm{div} (Y_j \mathbf{1}_E)- \mathbf{1}_E\mathrm{div} Y_j =0 
	\]
	in the sense of distributions over $\Omega$, and the $d\times d$ matrix  $Y(x)$, the rows of which are vectors $Y_j(x)$, is invertible
	for a.e.\ $x\in \Omega$ with the matrix function
	$Y^{-1}\colon \Omega\to \R^{d\times d}$ satisfying $Y^{-1}\in  W^{1,s}_{loc} (\Omega;\R^{d\times d})$ for some $s>\hat q'$,
    where $\hat q':= q^*\wedge r$,
	then a.e.\ over $\Omega$ one has
\begin{align*}
	\mbox{either }  &\mathbf{1}_E= 1,\\ 
	\mbox{or }  &\mathbf{1}_E= 0, 
\end{align*}
the first option being valid necessarily when $\mathcal{L}^d(B\cap\Omega)>0$.
	In particular, if $\Omega=\R^d$, then $\mathbf{1}_E= 1$ a.e., that is, \ $\mathcal{L}^d(\R^d\setminus E)=0$. 
\end{lemma}

\begin{proof}
	Let $u_k$ be smooth functions with $u_k\to \mathbf{1}_E$ in $L^p_{loc}(\R^d)$ as $k\to\infty$
 with 	$1/p\leq 1-1/\hat q-1/s$, $p<+\infty$ (it is exactly here that we use the strict inequality $s>\hat q'$).
	For instance, $u_k$ can be obtained from $\mathbf{1}_E$ by means of the customary convolution with an approximate identity. Then
	for
	\begin{align*}
		r_k^j &:= \mathrm{div} u_k Y_j - u_k\mathrm{div} Y_j
	\end{align*}
	one has $r_k^j= Y_j\cdot \nabla u_k$, and so $\nabla u_k =Y^{-1}  r_k$ where $r_k$ stands for the vector function
	$r_k(x):=(r_k^1(x),\ldots, r_k^d(x))^T$. For every test vector field $\Phi\in C_0^\infty (\Omega;\R^d)$ one has
	$\Psi:=\Phi^T Y^{-1}\in W^{1,s}(\Omega;\R^d)$, and therefore 
	\begin{align*}
		\langle \Phi, \nabla u_k \rangle &= \int_{\Omega} \Phi(x)\cdot \nabla u_k (x)\, dx =  \int_{\Omega} \Phi(x)^T Y^{-1}(x) r_k(x)\, dx
		\\ 
		&=
		\int_{\Omega} \Psi(x)\cdot r_k(x)\, dx \to 0 
	\end{align*}
	as $k\to\infty$ by Lemma~\ref{lm_invE2a_syst1}. 
	This means $\nabla u_k\rightharpoonup 0$ in the sense of distributions over $\Omega$, but since also $\nabla u_k\rightharpoonup \nabla \mathbf{1}_E$
	in the same sense, we get $\nabla \mathbf{1}_E=0$ as distributions over $\Omega$, i.e. $\mathbf{1}_E$ is constant, hence either $0$ or $1$ (up a.e.\ equality)
	 over $\Omega$.
	 If $\mathcal{L}^d(B\cap\Omega)>0$, then recalling that $B\subset E$ (again, up to a Lebesgue negligible set), then we have
	  $\mathbf{1}_E=1$ a.e.\ over $\Omega$, 
	concluding the proof.
\end{proof}

\subsection{Proof of Theorem~\ref{th_chow1weak} and Remark~\ref{rm_locHormander1}}

Now we are able to prove Theorem~\ref{th_chow1weak} and simultaneously the more general (though also more technical)  Remark~\ref{rm_locHormander1}. 

\begin{proof}
By Lemma~\ref{lm_invE1}, the set $E:=E(B)$ is invariant with respect to the flow of each $X_k$
which is flow generating. By Lemma~\ref{lm_invE2} therefore the function $\mathbf{1}_E$ satisfies the transport equations~\eqref{eq_X1Ea} with all  $X_k$ involved in the construction of $Y_j$ in place of $X$. By Lemma~\ref{lm_invE4}
it also satisfies the same transport equation with any of the consecutive Lie brackets of the vector fields $X_k$ involved in $Y_j$.
We may finally invoke then Lemma~\ref{lm_systtransp1b} with $\Omega:=\R^d$ to get the claim~(i).
Alternatively, if we are in slightly more general conditions of the localized Sobolev-H\"{o}rmander condition of Remark~\ref{rm_locHormander1},
we get from  Lemma~\ref{lm_systtransp1b} that $\mathbf{1}_E=1$ for some $\Omega\in \mathcal{A}$ (in particular for those that contain a piece of $B$ of positive Lebesgue measure). If there is an $\Omega\in \mathcal{A}$ such that 
 $\mathbf{1}_E=0$ over this $\Omega$ (another possible option by Lemma~\ref{lm_systtransp1b}), then $\R^d$ is a union of two disjoint open nonempty sets, one being the union of those $\Omega\in \mathcal{A}$ such that 
 $\mathbf{1}_E=1$ over $\Omega$, another the union of those $\Omega\in \mathcal{A}$ such that 
 $\mathbf{1}_E=0$ over $\Omega$. This contradicts the connectedness of $\R^d$, hence proving again the claim~(i). 
The claim~(ii) holds since otherwise $\mathbf{1}_{E(B_\varepsilon(x)))} =0$ over $B_\varepsilon(y)$ and hence
$\mathcal{L}^d(E(B_\varepsilon(x))\cap B_\varepsilon(y))=0$ contradicting claim~(i).
\end{proof}

%
%

%
%
%

\appendix

\section{Auxiliary lemmata}

We collect here several technical statements of general nature used throughout the paper.

	\begin{lemma}\label{lm_auxineq1}
		If $f,g\colon \R^d\to \R$ are Borel functions and $f\leq g$, then $T_*f\leq T_*g$ when 
		$T\colon \R^d\to \R^d$ is a Borel map satisfying
		$T_\#\mathcal{L}^d\ll \mathcal{L}^d$.
	\end{lemma}
	
	\begin{proof}
		One has
	\begin{align*}	
	\mathcal{L}^d \left(\left\{ x\in \R^d\colon  f(T(x))> g(T(x))\right\} \right) &\leq  	\mathcal{L}^d \left(T^{-1}\left(\left\{ y\in \R^d\colon  f(y) > g(y)\right\}\right)\right)=0,
	\end{align*}
	since $\mathcal{L}^d \left(\left\{ y\in \R^d\colon  f(y) > g(y)\right\}\right)=0$.
	\end{proof}
	


\begin{lemma}\label{lm_LnC2nonsmooth}
Let $V\in L^1_{loc}(\R^d;\R^d)$ 
be a vector field 
satisfying
\[
|V(x)|\leq \alpha |x|+\beta  \quad\mbox{for a.e.\ $x\in \R^d$}
\]
and
generating the flow $e^{tV}$.
If
	$\mathrm{div} V\in W^{1,1}_{loc}(\R^d)\cap L^\infty_{loc} (\R^d)$, 
	 then for
	$\mu_t:= e^{tV}_{\#}\mathcal{L}^d = \rho\, dx$, where
	$\rho=\rho(t,x)$ (in particular, $\mu_0=\mathcal{L}^d$, $\rho(0,x)=1$), and for any
bounded Borel function $f\colon \R\times \R^d\to \R^d$ such that the supports 
of  $f(t,\cdot)$ for all $t\in (-T,T)$ belong to some compact
set, one has
	\[
	\begin{aligned}
		\int_{\R^d} f(t,x) \, d\mu_t(x) &= \int_{\R^d} f(t,x) \, dx - t \int_{\R^d} f(t,x) \mathrm{div} V (x)\, dx + R(t) ,
	\end{aligned}
\]
for $t\in (-T,T)$, where $|R(t)|\leq C\|f\|_\infty t^2/2$
for some $C>0$ independent on $t$.
\end{lemma}

\begin{proof}
Approximate $V$ by convolution with a compactly supported  approximate identity by smooth vector fields 
$V_k\colon \R^d\to \R^d$ as in Lemma~\ref{lm_Kprime_est1}.
Let $K\subset\R^d$ be a compact set containing the supports 
of  $f(t,\cdot)$ for all $t\in (-T,T)$. 
By Lemma~\ref{lm_Kprime_est1} 
\begin{equation}\label{eq_estK1}
e^{tV_k}(K)\subset K' 
\end{equation}
for every $t\in (-T,T)$ with $K'$ depending only on $\alpha$, $\beta$, $\mbox{diam}({\{0\}\cup K})$ and $T$.
One has then
\begin{equation}\label{eq_estnormsV1}
\begin{aligned}
	\|V_k\|_{\infty; K'} & \leq 	\|V\|_{\infty; (K')_1},  \quad  \|\mathrm{div} V_k\|_{\infty; K'} \leq 	\|\mathrm{div} V\|_{\infty; (K')_1},\\
\|\nabla \mathrm{div} V_k\|_{1; K'}  &\leq 	\|\nabla \mathrm{div} V\|_{1; (K')_1}
\end{aligned}
\end{equation}
and $V_k\to V$ in $L^\mathbf{1}_{loc}(\R^d;\R^d)$, $\mathrm{div}V_k\to \mathrm{div}V$  in $L^\mathbf{1}_{loc}(\R^d)$ as $k\to\infty$. 
Let 
\[
\mu_t^k:= e^{tV_k}_{\#}\mathcal{L}^d = \rho^k\, dx, \quad\mbox{where }\rho^k=\rho^k(t,x). 
\]

Let $t\in (-T,T)$. 
For a $x\in K$ in view of~\eqref{eq_estK1} and~\eqref{eq_estnormsV1} we have 
\begin{equation}\label{eq_estexp1a}
	e^{-\int_0^t \mathrm{div}\, V_k(e^{-sV_k}(x))\, ds}\leq e^{T\|\mathrm{div}\, V_k\|_{\infty, K'}} \leq
	A:=e^{T\|\mathrm{div}\, V\|_{\infty, (K')_1}}.
\end{equation}
From
\[
\rho^k(t,x) - 1 =\rho^k(t,x) -  \rho^k(0,x) =\int_0^t \rho^k_{t}(\tau,x) \, d\tau
\]
applying Lemma~\ref{lm_LnC2smooth} with $V_k$ and $\rho^k$ instead of $V$ and $\rho$ respectively and using~\eqref{eq_estexp1a}, we get
\begin{equation}\label{eq_estexp2}
	\begin{aligned}
	\|\rho^k(t,\cdot)\|_{\infty, K} &\leq 1 + 	T\sup_{\tau\in [0,t]}\|\rho^k_t(\tau,\cdot)\|_{\infty, K}\\
	& \leq 1+  AT\|\mathrm{div} V_k\|_{\infty, K'} \leq 
		B:=1+  AT\|\mathrm{div}V\|_{\infty, (K')_1}.
\end{aligned}
\end{equation}
Using again Lemma~\ref{lm_LnC2smooth} with $V_k$ and $\rho^k$ instead of $V$ and $\rho$, we get
with the help of~\eqref{eq_estexp1a} the estimate
\begin{equation*}\label{eq_rhott1}
\begin{aligned}
			|\rho_{tt}^k(t,x)| &\leq 
	A \left(
	\left(\mathrm{div}\, V_k(e^{-tV_k}(x))\right)^2+ \left| \nabla \mathrm{div}\, V_k(e^{-tV_k}(x))\right| \cdot\left| V_k(e^{-tV_k}(x))
	\right|
	\right).
	\end{aligned}
\end{equation*} 
Hence
\begin{equation}\label{eq_rhott2a}
	\begin{aligned}
		\int_K|\rho_{tt}^k(t,x)|\, dx &\leq 
		A \int_{e^{tV_k}(K)} \left(
		\left(\mathrm{div}\, V_k(x)\right)^2+ \left| \nabla \mathrm{div}\, V_k((x)))\right| \cdot\left| V_k((x))
		\right|
		\right)\, dx\\
		&\leq
		A \int_{K'} \left(
		\left(\mathrm{div}\, V_k(x)\right)^2+ \left| \nabla \mathrm{div}\, V_k((x)))\right| \cdot\left| V_k((x))
		\right|
		\right)\rho^k(t,x)\, dx\\
		& \qquad\qquad\qquad\qquad\qquad\qquad\qquad\qquad\mbox{by~\eqref{eq_estK1}}\\
		&\leq C:=AB\left(\|\mathrm{div}\, V\|_{\infty, (K')_1}^2+ \|\nabla \mathrm{div}\, V\|_{1, (K')_1} \cdot \|V\|_{\infty, (K')_1}\right)\\
		& \qquad\qquad\qquad\qquad\qquad\qquad\qquad\qquad\mbox{by~\eqref{eq_estnormsV1} and~\eqref{eq_estexp2}}	.
	\end{aligned}
\end{equation} 

Taylor's theorem with integral form of the remainder yields
\[
\rho^k(t,x) -  \rho^k(0,x) -t \rho^k_t(0,x)=\int_0^t \rho^k_{tt}(\tau,x)(t-\tau) \, d\tau,
\]
so that
\begin{equation}\label{eq_estdensTaylor1a}
	\begin{aligned}
	\int_{\R^d} f(t,x) \rho^k(t,x) \, dx &- \int_{\R^d} f(t,x) \rho^k(0,x) \, dx 
		-t\int_{\R^d} f(t,x)  \rho_t^k(0,x) \, dx\\
		& =  \int_{\R^d}f(t,x) \, dx\int_0^t 	\rho^k_{tt}(\tau,x)(t-\tau) \, d\tau
	\end{aligned}
\end{equation}
We estimate now 
\begin{equation}\label{eq_estdensTaylor1b}
	\begin{aligned}
 \left|\int_{\R^d}f(t,x) \right. & \, dx  \left. \int_0^t 	\rho^k_{tt}(\tau,x)(t-\tau) \, d\tau \right| \\
 & \leq \int_0^t (t-\tau)\, d\tau \int_{\R^d} |f(t,x)| \, dx 	|\rho^k_{tt}(\tau,x)|\\
 &\leq \| f\|_\infty \int_0^t (t-\tau)\, d\tau \int_{K} 	|\rho^k_{tt}(\tau,x)|\, dx \\
 &\leq C\| f\|_\infty \int_0^t (t-\tau)\, d\tau   
 \quad\mbox{by~\eqref{eq_rhott2a}}\\
 &\leq  C\| f\|_\infty t^2/2 .
	\end{aligned}
\end{equation}
Plugging~\eqref{eq_estdensTaylor1b} into~\eqref{eq_estdensTaylor1a}
and recalling that $\rho^k(0,x)=1$ and $\rho^k_t(0,x)=-\mathrm{div} V^k$, we get
\[
\begin{aligned}
	\left|\int_{\R^d} f(t,x) \, d\mu_t^k(x) - \int_{\R^d} f(t,x) \, dx +t\int_{\R^d} f(t,x) \mathrm{div} V^k (x)\, dx\right|  \leq 
	C\| f\|_\infty t^2/2.
\end{aligned}
\]
Passing to the limit as $k\to \infty$ in the above inequality, in view of Definition~\ref{def_flowgen1}(v) we get
\[
\begin{aligned}
	\left|\int_{\R^d} f(t,x) \, d\mu_t(x) - \int_{\R^d} f(t,x) \, dx+ t\int_{\R^d} f(t,x) \mathrm{div} V (x)\, dx\right|  \leq	C\| f\|_\infty t^2/2,
\end{aligned}
\]
which is the claim being proven.
 \end{proof}

The following  lemmata on smooth vector fields have been used in the proof of the above Lemma~\ref{lm_LnC2nonsmooth}.

	\begin{lemma}\label{lm_LnC2smooth}
	If $V\colon \R^d\to\R^d$ is a smooth vector field, then for
	$\mu_t:= e^{tV}_{\#}\mathcal{L}^d = \rho\, dx$, where
	$\rho=\rho(t,x)$ $($in particular, $\mu_0=\mathcal{L}^d$, $\rho(0,x)=1)$, one has
	\[
	\begin{aligned}
		\rho_{t}(t,x) &=  
		-e^{-\int_0^t \mathrm{div}\, V(e^{-sV}(x))\, ds}
		\mathrm{div}\, V(e^{-tV}(x)),\\
		\rho_{tt}(t,x) &=  \\
        		&e^{-\int_0^t \mathrm{div}\, V(e^{-sV}(x))\, ds}
		\left(
		\left(\mathrm{div}\, V(e^{-tV}(x))\right)^2-\nabla \mathrm{div}\, V(e^{-tV}(x)) \cdot V(e^{-tV}(x))
		\right).
	\end{aligned}
	\]
\end{lemma}

\begin{proof}
	One has 
	\begin{equation}\label{eq_rhot1}
		1=\rho(0,x)=\rho(t, e^{tV}(x))|\det D e^{tV}(x)|.
	\end{equation}
	But the matrix $Y(t):= D e^{tV}(x)$ satisfies the ODE
	\[
	\dot{Y}= A(t) Y, 
	\]
	where $A(t):= DV(e^{tV}(x))$ and the initial condition $Y(0)=\mathrm{Id}$ where $\mathrm{Id}$ stands for the $d\times d$ identity matrix.
	Thus the Ostrogradski\v{\i}-Liouville formula gives
	\[
	\det Y(t)=\exp\left(
	{\int_0^t \mathrm{div}\, V(e^{sV}(x))\, ds}
	\right),
	\]
	so that from~\eqref{eq_rhot1} we get
	\[
	\rho(t,e^{tV}(x)) = \exp\left(
	{-\int_0^t \mathrm{div}\, V(e^{sV}(x))\, ds}
	\right).
	\]
	Letting $y:=e^{tV}(x)$ in the above formula we get
	\[
	\rho(t,y) = \exp\left(
	{-\int_0^t \mathrm{div}\, V(e^{(s-t)V}(y))\, ds} 
	\right) =
	\exp\left(
	{-\int_0^t \mathrm{div}\, V(e^{-\tau V}(y))\, d\tau} 
	\right),
	\]
	and the statement follows from just taking derivatives in $t$ of the latter expression.
\end{proof}

\begin{lemma}\label{lm_Kprime_est1}
	Let $V\in L^1_{loc} (\R^d;R^d)$ be a  vector field satisfying
	\[
	|V(x)|\leq \alpha |x|+\beta  \quad\mbox{for a.e.\ $x\in \R^d$}
	\]
	for some $\alpha >0$, $\beta>0$.
	Approximate $V$ by convolution with a compactly supported  approximate identity by smooth vector fields 
	$V_k\colon \R^d\to \R^d$, namely, let 
	\[ V_k := V*\varphi_{1/k}, \quad\mbox{where }
	\varphi_\varepsilon (x):= \frac{1}{\varepsilon^d}\varphi\left( \frac{|x|}{\varepsilon}\right)), 
	\]
and $\varphi\in C_0^\infty(\R)$ be such that
$\varphi(t)>0$ for every $t\in (-1,1)$, and 
$\varphi(t)=0$ otherwise, while $\int_\R\varphi(t)\, dt=1$.
	Then for all $k\geq 1$ one has
\begin{equation}\label{eq_modVk1}
|V_k(x)|\leq \tilde \alpha |x|+\tilde \beta  \quad\mbox{for all $x\in \R^d$}
\end{equation}
for some $\tilde \alpha >0$, $\tilde \beta>0$ depending only on $\alpha$ and $\beta$ and
for every compact set $K\subset \R^d$ and every $T>0$ the estimate holds
\[
e^{tV_k}(K)\subset K' 
\]	
for every $t\in (-T,T)$ with $K'$ depending only on $\alpha$, $\beta$, $\mbox{diam}({\{0\}\cup K})$ and $T$.
\end{lemma}

\begin{proof}
One has
\begin{align*}
|V_k(x)| &\leq (|V|*\varphi_{1/k}) (x) \leq \alpha \int_{\R^d} |x-y| \varphi_{1/k} (y)\, dy +\beta\\
&\leq  \alpha \left( |x|+ \frac{1}{k} \right)\int_{\R^d}\varphi_{1/k} (y)\, dy +\beta\\
&= \alpha\left( |x|+ \frac{1}{k} \right) +\beta \leq \alpha |x|+ (\alpha +\beta).
\end{align*}	
showing~\eqref{eq_modVk1}. 
The last claim about $e^{tV_k}(K)$ is then just Lemma~\ref{lm_Kprime_est2} below applied with $V_k$ instead of $V$, $\tilde \alpha:=\alpha$ instead of $\alpha$ and $\tilde \beta:=\alpha+\beta$ instead of $\beta$.
%
\end{proof}

We also need the following immediate lemma.

\begin{lemma}\label{lm_Kprime_est2}
	Let $V\colon \R^d\to \R^d$ be a not necessarily 
	smooth vector field satisfying
	\[
	|V(x)|\leq \alpha |x|+\beta  \quad\mbox{for all\ $x\in \R^d$}
	\]
	for some $\alpha >0$, $\beta>0$. Suppose that $x(\cdot)\in \R^d$ be a solution to the ODE
	$\dot{x}=V(x)$ with $x(0)\in K$, where $K\subset \R^d$ be some compact set.
	 Then for every $T>0$ there is a compact set $K'\subset \R^d$ depending only on $\alpha$, $\beta$, $\mbox{diam}({\{0\}\cup K})$ and $T$
	 such that $x(t)\in K'$ for all $t\in (-T,T)$.
\end{lemma}

\begin{proof}
Setting $\rho(t):=|x(t)|$, we get
	$\dot{\rho} \leq\alpha\rho+\beta$ with $\rho(0)= |x(0)|\leq \mbox{diam}({\{0\}\cup K})$. Thus
	\[
	\rho(t)\leq ((\alpha \rho(0)+\beta)e^{\alpha T}-\tilde\beta)/\tilde\alpha\leq 
	R:=((\alpha \mbox{diam}({\{0\}\cup K})+\beta)e^{\alpha T}-\beta)/\alpha,
	\]
	so that we may take $K':=\bar B_R(0)$.
\end{proof}

The final lemma proves the formula for the divergence of a Lie bracket of two Sobolev vector fields known in the classical case (when the vector
fields are smooth).

\begin{lemma}\label{lm_divLie1}
	If for some $q\in (1,+\infty)$ one has
	$Y\in W^{1,q'}_{loc}(\R^d;\R^d)$ with $\mathrm{div} Y \in W^{1,q'}_{loc}(\R^d)$, and $X\in W^{1,q}_{loc}(\R^d;\R^d)$
	with $\mathrm{div} X \in W^{1,q}_{loc}(\R^d)$, then 
		\[
	\mathrm{div}[X,Y] = X\cdot \nabla \mathrm{div}Y- Y\cdot \nabla \mathrm{div}X,
	\]
\end{lemma}

\begin{proof}
	It suffices to prove
	\begin{equation}\label{eq_wkdivLie1}
		\int_{\R^d} \nabla \varphi\cdot [X,Y]\, dx = \int_{\R^d}\left(X\cdot \nabla \mathrm{div}Y- Y\cdot \nabla \mathrm{div}X\right)\varphi\,dx 
	\end{equation}
	for all $\varphi\in C_0^\infty(\R^d)$.
	To this aim approximate both $X$ and $Y$ by convolution with  approximate identity by smooth vector fields 
	$X_k\colon \R^d\to \R^d$, $Y_k\colon \R^d\to \R^d$
	so that 
	\begin{align*}
		X_k\to X,&\quad\mbox{in $L^q_{loc}(\R^d;\R^d)$},\qquad
		Y_k\to Y \quad\mbox{in $L^{q'}_{loc}(\R^d;\R^d)$},\\
		DX_k\to DX&\quad\mbox{in $L^q_{loc}(\R^d; \R^{d\times d})$},\qquad
		DY_k\to DY \quad\mbox{in $L^{q'}_{loc}(\R^d; \R^{d\times d})$},\\
	    \nabla\mathrm{div} X_k\to \nabla\mathrm{div}X &\quad\mbox{in $L^q_{loc}(\R^d;\R^d))$},\qquad
		\nabla\mathrm{div} Y_k\to \nabla\mathrm{div}Y \quad\mbox{in $L^{q'}_{loc}(\R^d;\R^d))$}
	\end{align*}
	as $k\to\infty$. Then~\eqref{eq_wkdivLie1} follows from passing to the limit as $k\to\infty$ in the identity
\begin{equation*}
		\int_{\R^d} \nabla \varphi\cdot [X_k,Y_k]\, dx = \int_{\R^d}\left(X_k\cdot \nabla \mathrm{div}Y_k- Y\cdot \nabla \mathrm{div}X_k\right)\varphi\,dx, 
	\end{equation*}
	concluding the proof.
\end{proof}

%
%
%
%
%

\end{document}